\documentclass[a4paper,reqno,11pt,oneside]{amsart}

\usepackage{amsfonts,amssymb,latexsym,xspace,epsfig,graphics,color}
\usepackage{amsmath,enumerate,stmaryrd,xy,amsthm}
\usepackage{a4wide}
\usepackage{color}
\usepackage{float}
\usepackage{graphicx}
\usepackage{cite}
\usepackage[latin1]{inputenc}

\usepackage[ruled,vlined]{algorithm2e}
\usepackage{caption}
\usepackage{subcaption}
\usepackage[normalem]{ulem}
\newtheorem{teo}{Theorem}[section]
\newtheorem{lema}{Lemma}[section]
\newtheorem{defn}{Definition}[section]

\newtheorem{no}{Remark}[section]

\makeatletter
\@namedef{subjclassname@2020}{%
  \textup{2020} Mathematics Subject Classification}
\makeatother

\usepackage{tikz}
\usepackage{tkz-berge}
\usepackage{tkz-graph}
\usetikzlibrary{arrows,snakes,backgrounds,trees,positioning}
\usetikzlibrary{decorations.markings}
\usepackage{stmaryrd}
\tikzstyle directed=[postaction={decorate,decoration={markings,
    mark=at position .65 with {\arrow{stealth}}}}]
\tikzstyle reverse directed=[postaction={decorate,decoration={markings,
    mark=at position .65 with {\arrowreversed{stealth};}}}]

\title[]{On a fractional queueing model with catastrophes}

\author{Matheus de Oliveira Souza and Pablo M. Rodriguez}

\date{}

\address{
\newline
Matheus de Oliveira Souza
\newline
Instituto De Ci\^encias Matem\'aticas e de Computa\c c\~ao
\newline
Av. Trabalhador S\~ao Carlense, 400 - Centro, CEP 13566-590, S\~ao Carlos - SP, Brazil.
\newline
e-mail: matheus.oliveira.souza@usp.br
\newline
\newline
Pablo M. Rodriguez
\newline
Centro de Ci\^encias Exatas e da Natureza, Universidade Federal de Pernambuco
\newline
Av. Prof. Moraes Rego, 1235. Cidade Universit\'aria, CEP 50670-901, Recife - PE, Brazil.
\newline
e-mail:  pablo@de.ufpe.br
}

\thanks{}

\begin{document}
  
\maketitle

\begin{abstract}
A $M/M/1$ queue with catastrophes is a modified $M/M/1$ queue model for which, according to the times of a Poisson process, catastrophes occur leaving the system empty.  In this work, we study a fractional $M/M/1$ queue with catastrophes, which is formulated by considering fractional derivatives in the Kolmogorov's Forward Equations of the original Markov process. For the resulting fractional process, we obtain the state probabilities, the mean and the variance for the number of customers at any time. In addition, we discuss the estimation of parameters.
\end{abstract}

\section{Introduction}
Queueing Theory allows the formulation of mathematical models and methods to deal with stochastic aspects in applied sciences. Roughly speaking, the models are stochastic processes, usually Markovian, to represent phenomena in which customers arrive in a random way at a service facility. Upon arrival, they are made to wait in queue until it is their turn to be served and after that, it is assumed that they leave the system. The main interest of such models is in the behavior of the number of customers in the system at any time. This could be studied by stating the state probabilities, mean and variance, between other quantities. For a friendly introduction to Queueing Theory, we refer the reader, for instance, to \cite[Chapter 8]{ross}.

A class of well-known queueing models is the exponential one, where it is assumed that arrivals occur according to a Poisson process, and it is assumed that each service time follows an exponential law. Such models are usually called $M/M/k$ queues, where $k$ represents the number of servers. Our interest is in the $M/M/1$ queue with catastrophes in which, additionally, it is supposed that, according to the times of a Poisson process, catastrophes occur leaving the system empty. This model has been studied, for instance, in \cite{catastrofef,catastrofe,kumar}. Here we shall consider a non-Markovian version of such a model, which is inspired by a series of modifications in probabilistic models that appeared as a consequence of the development of the fractional calculus. The recent interest in fractional calculus has been increased by its applications, mainly in numerical analysis and different areas of physics, engineering, economy, etc. In Probability Theory, the fractional calculus, combined with stochastic processes, is useful to represent random phenomena with long memory; that is, where the Markov property does not apply. This is the reason why in recent years the efforts of a large number of researchers have been directed towards the formulation of the fractional counterpart of classical Markovian models. For some examples, we refer the reader to \cite{catastrofef,yule,death,fila,cahoy2010, meerschaert-2019,  puro} and the references therein.


Our purpose is to contribute to this effort by studying the fractional version of the $M/M/1$ queue with catastrophes. Fractional queues were studied for the first time by \cite{fila}, where the authors proposed a generalization of the classical $M/M/1$ queue model derived by applying fractional derivative operators to the Kolmogorov's Forward Equations of the original process. The approach proposed by \cite{fila} allows the formulation of closed-expression for some functional of interest, and at the same time, the estimation of parameters. As far as we know,
the fractional version of the $M/M/1$ queue with catastrophes was proposed and studied only by \cite{catastrofef}, where the authors provided expressions for the state probabilities, the distributions of the busy period for fractional queues without and with catastrophes and the distribution of the time of the first occurrence of a catastrophe. In this work we complement the analysis of \cite{catastrofef} by appealing to the approach proposed by \cite{fila}. As a contribution to the field, we provide a closed-expression for the state probabilities, assuming that the process starts from any state, and for the mean and the variance for the number of customers at any time. In addition, we deal with the estimation of parameters for the model, and we illustrate our results with computational simulations.



We organize the paper as follows: in Section 2 we introduce the classical $M/M/1$ queueing model with catastrophes and its fractional generalization using a subordination relationship. Furthermore, we obtain the state probabilities, and we state the mean and the variance for the number of customers at any time by using a probability generating function. In Section 3, we present the estimation of parameters for the model and their confidence intervals. Lastly, we summarize our results in a brief conclusion in Section 4.

\section{The model and results}

Let $\{X_t\}_{t\geq 0}$ be the exponential queue model $M/M/1$ with catastrophes, and let $P_{i,n}(t):=P(X_t=n|X_0=i)$ be its transition probabilities, where $i,n$ are non-negative integers. In other words, assume that customers arrive at a single-server service system according with a Poisson process of parameter $\lambda>0$. Upon arrival, each customer is made to wait in a unique queue until it is his/her turn to be served. If the server is free at the arrival of a customer then he/she goes directly into service. After service is complete, the corresponding customer leaves the system, and the next customer in the queue enters service. It is assumed that the sequence elements of service times are independent random variables with a common exponential law of parameter $\mu>0$. Moreover, it is assumed that according to the times of a Poisson process of parameter $\xi\geq 0$ catastrophes occur leaving the system empty. For any $t\geq 0$ the random variable $X_t$ denotes the number of customers in the system at time $t$, and $\{X_t\}_{t\geq 0}$ is the continuous-time Markov chain with transitions given by:

$$P_{i,n}(h)=\left\{
\begin{array}{cl}
\lambda h + o(h),& \text{ if }i\in\mathbb{N}\cup\{0\}\text{ and }n=i+1,\\[.2cm]
\mu h + o(h),& \text{ if }i\in\mathbb{N}\setminus \{1\}\text{ and }n=i-1,\\[.2cm]
(\mu + \xi) h + o(h),& \text{ if }i=1\text{ and }n=0,\\[.2cm]
\xi h + o(h),& \text{ if }i\in\mathbb{N}\setminus\{1\}\text{ and }n=0,\\[.2cm]
\end{array}\right.
$$
where $o(h)$ represents a function such that $\lim_{h\to 0}o(h)/h =0$ (see Figure \ref{fig:trans}).

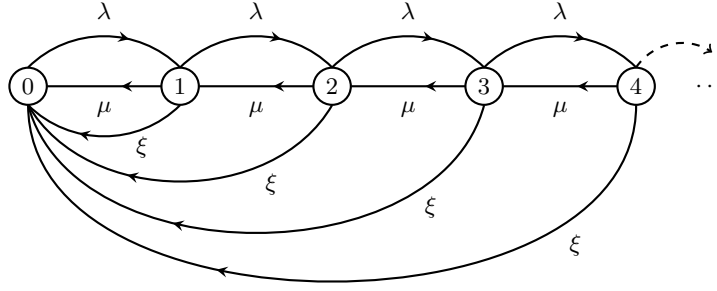
\begin{figure}[h!]
    \centering

\begin{tikzpicture}

\draw [thick] (0,0) circle (7pt);
\draw (0,0) node[font=\footnotesize] {$0$};
\draw [thick] (2,0) circle (7pt);
\draw (2,0) node[font=\footnotesize] {$1$};
\draw [thick] (4,0) circle (7pt);
\draw (4,0) node[font=\footnotesize] {$2$};
\draw [thick] (6,0) circle (7pt);
\draw (6,0) node[font=\footnotesize] {$3$};
\draw [thick] (8,0) circle (7pt);
\draw (8,0) node[font=\footnotesize] {$4$};


\draw [thick, directed] (0,0.25) to [bend left=45] (2,0.25);
\draw [thick, directed] (2,-0.25) to [bend left=45] (0,-0.25);
\draw [thick, directed] (4,-0.25) to [bend left=60] (0,-0.25);
\draw [thick, directed] (6,-0.25) to [bend left=75] (0,-0.25);
\draw [thick, directed] (8,-0.25) to [bend left=90] (0,-0.25);
\draw [thick, directed] (2,0.25) to [bend left=45] (4,0.25);
\draw [thick, directed] (4,0.25) to [bend left=45] (6,0.25);
\draw [thick, directed] (6,0.25) to [bend left=45] (8,0.25);
\draw [->, dashed, thick] (8,0.25) to [bend left=45] (9,0.45);
\draw [thick, reverse directed] (4.25,0) to (5.75,0);
\draw [thick, reverse directed] (2.25,0) to (3.75,0);
\draw [thick, reverse directed] (0.25,0) to (1.75,0);
\draw [thick, reverse directed] (6.25,0) to (7.75,0);


\draw (1,1) node[font=\footnotesize] {$\lambda$};
\draw (1.5,-0.8) node[font=\footnotesize] {$\xi$};
\draw (3.2,-1.3) node[font=\footnotesize] {$\xi$};
\draw (5.3,-1.6) node[font=\footnotesize] {$\xi$};
\draw (7.2,-2.1) node[font=\footnotesize] {$\xi$};
\draw (3,1) node[font=\footnotesize] {$\lambda$};
\draw (5,1) node[font=\footnotesize] {$\lambda$};
\draw (7,1) node[font=\footnotesize] {$\lambda$};
\draw (5,-0.3) node[font=\footnotesize] {$\mu$};
\draw (1,-0.3) node[font=\footnotesize] {$\mu$};
\draw (3,-0.3) node[font=\footnotesize] {$\mu$};
\draw (7,-0.3) node[font=\footnotesize] {$\mu$};
\draw (9,0) node[font=\footnotesize] {$\cdots$};

\end{tikzpicture}

    \caption{Transitions and rates for the $M/M/1$ queue with catastrophes.}
    \label{fig:trans}
\end{figure}

The Kolmogorov's Forward Equations for the exponential queue model $M/M/1$ with catastrophes are given by

\begin{equation}
     \begin{cases}
 \displaystyle        \frac{\partial P_{i,0}(t)}{\partial t}=-(\lambda + \xi)P_{i,0}(t)+\mu P_{i,1}(t)+\xi, \mbox{ $t > 0$, $i \geq 0$,}\\[0.4cm]
 \displaystyle        \frac{\partial P_{i,n}(t)}{\partial t}=-(\lambda + \mu + \xi)P_{i,n}(t)+\lambda P_{i,n-1}(t)+\mu P_{i, n+1}(t),\mbox{ $t > 0$, $i \geq 0$, $n \geq 1$,}\\[0.4cm]
          P_{i,n}(0)=\delta_{i, n},
    \end{cases}
    \label{eqdif}
\end{equation}

\noindent
where $\delta_{i, n}$ is the Kronecker delta defined by $\delta_{i, n}=1$ if $i=n$, or $\delta_{i, n}=0$, otherwise. Our purpose is to define a fractional version for the exponential queue model $M/M/1$ with catastrophes. 

\begin{defn}
A fractional $M/M/1$ queue model with catastrophes, with parameter $\alpha\in (0,1)$, is the continuous-time stochastic process $\{X^\alpha_t\}_{t\geq 0}$ such that the transition probabilities $P^\alpha_{i,n}(t):=P\{X^\alpha_t=n|X^\alpha_0=i\}$ satisfy
\begin{equation}\label{eqdiff}
    \begin{cases}
          D^\alpha _t P^\alpha_{i,0}(t)=-(\lambda + \xi)P^\alpha_{i,0}(t)+\mu P^\alpha_{i,1}(t)+\xi, \mbox{ $t > 0$, $i \geq 0$,}\\[0.4cm]
         D^\alpha _t P^\alpha_{i,n}(t)=-(\lambda + \mu + \xi)P^\alpha_{i,n}(t)+\lambda P^\alpha_{i,n-1}(t)+\mu P^\alpha_{i,n+1}(t),\mbox{ $t > 0$, $i \geq 0$, $n \geq 1$,}\\[0.4cm]
          P^\alpha_{i,n}(0)=\delta_{i,n},
    \end{cases}
\end{equation}
where 
$$D^\alpha_t P^\alpha_{i,n}(t):= \frac{1}{\Gamma(1 - \alpha)}\int^t_0\frac{dP^\alpha_{i,n}(s)/ds}{(t - s)^\alpha}ds$$
is the Caputo's fractional derivative of order $\alpha$, and $\Gamma(s)$ is the gamma function defined by 
$$\Gamma(s):=\int_{0}^{\infty}x^{s-1}e^{-x}dx.$$
\end{defn}


\begin{lema}Let $\{X^\alpha_t\}_{t\geq 0}$ be a fractional $M/M/1$ queue model with catastrophes, with parameter $\alpha$, and consider the probability generating function (p.g.f.) $$G^\alpha(z,t)=G_i^\alpha(z,t):=\sum^\infty_{n=0}z^nP^\alpha_{i, n}(t), \mbox{ $t > 0$, $i \geq 0$.}$$ 
Then,

\begin{equation}
    \begin{cases}
         \displaystyle zD^\alpha_tG^\alpha(z,t)=(1-z)\left[(\mu-z\lambda-\frac{\xi z}{(1-z)})G^\alpha(z,t)-\mu P^\alpha_{i,0}(t)+\frac{\xi z}{(1-z)}\right],  \\
         G^\alpha(z,0)=z^i.
         \label{sistema}
    \end{cases}
\end{equation}
\end{lema}

\begin{proof}

On the one hand, by \eqref{eqdiff}, we obtain
\begin{multline}\label{des_sistema}
D^\alpha _t[G^\alpha(z,t)-P^\alpha_{i,0}(t)]=-(\lambda+\mu+\xi)[G^\alpha(z,t)-P^\alpha_{i,0}(t)]+ \\
+\lambda zG^\alpha(z,t) +\frac{\mu}{z}[G^\alpha(z,t)-P^\alpha_{i,0}(t)-zP^\alpha_{i,1}(t)].
\end{multline}

On the other hand, by replacing $D^\alpha_t P^\alpha_{i,0}$ from \eqref{eqdiff} in \eqref{des_sistema} we get

\begin{equation}
\label{eqlem1}
D^\alpha _t[G^\alpha(z,t)-P^\alpha_{i,0}(t)]=D^\alpha _tG^\alpha(z,t)-[-(\lambda + \xi)P^\alpha_{i,0}(t)+\mu P^\alpha_{i,1}(t)+\xi].
\end{equation}

\smallskip
Putting \eqref{des_sistema} and \eqref{eqlem1} together we have

\begin{equation*}
\begin{array}{rl}
D^\alpha_tG^\alpha(z,t)=&-[\lambda(1-z) +\mu(1-1/z) + \xi]G^\alpha(z,t) +\mu(1 - 1/z)P^\alpha_{i,0}(t) +\xi,
\end{array}
\end{equation*}

\smallskip
which can be written as 

\begin{equation}\label{eq:lem2}
zD^\alpha_tG^\alpha(z,t)=[-z\lambda(1-z)+\mu(1-z)-z\xi]G^\alpha(z,t)-(1-z)\mu P^\alpha_{i,0}(t)+z\xi.
\end{equation}

Therefore, \eqref{sistema} is obtained by \eqref{eq:lem2} and by noting that $G^\alpha(z,0)= \displaystyle \sum^\infty_{n=0}z^n\delta_{i, n} = z^i.$
\end{proof}

\subsection{Representation of the fractional model, and state probabilities}

Our first task is to provide a representation for the fractional model through the exponential one, by mean of a time-change version of the original process. Let $\{B^\alpha_t\}_{t\geq 0}$ be the $\alpha$-stable subordinator and let $\{C^\alpha_t\}_{t\geq 0}$ its inverse process, where $C^\alpha_t$ is the first passage time to the level $t>0$, that is,
\begin{equation}
    C^\alpha_t:=\inf\{s>0,~B^\alpha_s>t\}.
\end{equation}
We refer the reader to \cite[Chapter 3]{bertoin} for more details about the subordinator and its inverse. In addition, we consider the Laplace transform of the inverse process 
\begin{equation}
    \int^\infty_0e^{-st}f_\alpha(y,t)dt=s^{\alpha-1}e^{ys^\alpha }.
\end{equation}

\noindent
Here $f_\alpha(y,t)$ is the density of $C_t^{\alpha}$, and it is given by 
\begin{equation}
    f_\alpha(y,t) = W_{-\alpha,1-\alpha}(-yt^\alpha)t^{-\alpha}=t^{-\alpha}\sum^\infty_{r = 0}\frac{(-yt^{-\alpha})^r}{r!\Gamma(1-\alpha(1+r))},
    \label{wright}
\end{equation}

\noindent
where $W_{-\alpha,1-\alpha}(-x)$, also denoted by $M_{\alpha}(x)$ in \cite{mainardi}, is the Wright distribution of parameter $\alpha$. Moreover, $f_\alpha(y,t)$ can be seen in the solution of the fractional diffusion equation. We refer the reader to \cite{mainardi, difusion} for more details.

\begin{teo}
\label{teosub}
Let $\{X_t\}_{t\geq 0}$ and $\{X^{\alpha}_t\}_{t\geq 0}$ be, respectively, the exponential and fractional with parameter $\alpha\in(0,1]$ queue model $M/M/1$ with catastrophes. If $\{C^\alpha_t\}_{t\geq 0}$, $\alpha \in (0,1]$, is the inverse $\alpha$-stable subordination process, and it is independent of $\{X_t\}_{t\geq 0}$, then
\begin{equation}
    X^\alpha_t=X_{C^\alpha_t}, \text{ for any }t\geq0
    \label{relacaosub}
\end{equation}
where the equality holds for the one-dimensional distribution.
\label{sub}
\end{teo}

\begin{proof}

We start by pointing out that \eqref{relacaosub} is equivalent to say that we can write the state probabilities as

\begin{equation}
    P^\alpha_{i,n}(t)=\int^\infty_0P_{i,n}(y)f_\alpha(y,t) dy,
    \label{resub}
\end{equation}

\smallskip
and the p.g.f. as
\begin{equation}
\begin{array}{rl}
    G^\alpha(z,t)&=\displaystyle\sum^\infty_{i = 0}z^i\left\{\int^\infty_0P_{i,n}(y)f_\alpha(y,t)dy\right\}\\[.5cm]
    &=\displaystyle\int^\infty_0\left\{\sum^\infty_{i = 0}z^iP_{i,n}(y)f_\alpha(y,t)\right\}dy\\[.5cm]
    &\displaystyle=\int^\infty_0G(z,y)f_\alpha(y,t)dy.
    \label{resub1}
    \end{array}
\end{equation}

Since we are interested in proving that we can rewrite the fractional process as a transformation of the exponential process, through the inverse $\alpha$-stable subordinator, then is enough if we prove that \eqref{resub} and \eqref{resub1} satisfy \eqref{sistema}. So applying the Laplace transform in \eqref{sistema}; namely, if $\tilde{G}^\alpha(z,s):=\int_{0}^{\infty}e^{-st}G^{\alpha}(z,t)dt,$ and $ \tilde{P}^\alpha_{i,n}(s) := \int^\infty_0e^{-st}P^\alpha_{i,n}(t)dt$, then 
\begin{equation}
    z[s^\alpha \tilde{G}^\alpha(z,s)-s^{\alpha-1}z^i]=(1-z)\left[\left(\mu-\lambda z-\frac{\xi z}{1-z}\right)\tilde{G}^\alpha(z,s)-\mu \tilde{P}^\alpha_{i,0}(s)\right]+\frac{\xi z}{s}.
    \label{trans1eq}
\end{equation}

By using \eqref{resub}, \eqref{resub1} and the Laplace transform of the inverse $\alpha$-stable subordinator process, we get that $z[s^\alpha \tilde{G}^\alpha(z,s)-s^{\alpha-1}z^i]$ is equal to
$$(1-z)\left[\left(\mu -\lambda z -\frac{\xi z}{1-z}\right)\int^\infty_0{G(z,y)s^{\alpha-1}e^{-ys^\alpha}dy} \displaystyle-\mu\int^\infty_0{P_{i,0}(y)}s^{\alpha-1}e^{-ys^\alpha}dy\right]+\frac{\xi z}{s},$$

so
\begin{multline}    \label{trans2eq}
    z[s^\alpha \tilde{G}^\alpha(z,s)-s^{\alpha-1}z^i]=\displaystyle \\
    (1-z)\int^\infty_0s^{\alpha-1}e^{-ys^\alpha}\left[\left(\mu -\lambda z -\frac{\xi z}{1-z}\right){G(z,y)}-\mu{P_{i,0}(y)}\right]dy+\frac{\xi z}{s}.
\end{multline}

Now, since

\begin{equation}
z\left[\frac{\partial G(z,t)}{\partial t}-\xi\right]=(1-z)\left[(\mu-z\lambda-\frac{\xi z}{(1-z)})G(z,t)-\mu P_{i,0}(t)\right],    
\end{equation}
we can use \eqref{trans2eq} to obtain

\begin{equation}\label{trans3eq}
\begin{split}
    z[s^\alpha \tilde{G}^\alpha(z,s)-s^{\alpha-1}z^i]&=\displaystyle s^{\alpha-1}\int^\infty_0{e^{-ys^\alpha}z\frac{\partial G(z,y)}{\partial y}}dy-s^{\alpha-1}\int^\infty_0{e^{-ys^\alpha}z\xi}dy+\frac{\xi z}{s}\\[.6cm]
   & =\displaystyle s^{\alpha-1}z\left[G(z,y)e^{-sy^\alpha} \Bigr|^{y=\infty}_{y=0}+s^\alpha\int^\infty_0{G(z,y)e^{-ys^\alpha}}dy\right]-\frac{\xi zs^{\alpha-1}}{s^\alpha}+\frac{\xi z}{s}\\[.6cm]
    &= s^{\alpha-1}z\left[s^\alpha\displaystyle \int^\infty_0{G(z,y)e^{-ys^\alpha}}dy-z^i\right]\\[.6cm]
    &= z[s^\alpha \tilde{G}^\alpha(z,y)-s^{\alpha-1}z^i].
    \end{split}
\end{equation}

\end{proof}

The previous theorem gains in interest if we realize that we can obtain the state probabilities of the fractional model provided we have the state probabilities of the exponential process. For the exponential $M/M/1$ queue with catastrophes $\{X_t\}_{t\geq 0}$, if 
$$I_n(z)=\sum^\infty_{m = 0}\frac{1}{m!\Gamma(m + n + 1)}\left(\frac{z}{2}\right)^{2m + n}, \text{where } n,z\in\mathbb{C}$$
denotes the modified Bessel function of the first kind, it is well-known (see \cite{kumar}) that   

    \begin{multline}\label{p1}
        P_{i,0}(t)=\\\frac{1}{\mu}\sum ^\infty_{n=i}\frac{(n+1)I_{n+1}((2\sqrt{\lambda\mu})t)e^{-(\lambda+\mu+\xi)t}}{(\sqrt{\lambda/\mu})^{n+1}t}+\frac{\xi}{\mu}\int^t_0\sum^\infty_{n=1}\frac{nI_{n}((2\sqrt{\lambda\mu})u)e^{-(\lambda+\mu+\xi)u}}{(\sqrt{\lambda/\mu})^{n}u}du
    \end{multline}

and, for $n > 0$
    \begin{equation}
    \begin{array}{ll}
        \displaystyle P_{i,n}(t)=&\displaystyle\frac{\xi(\sqrt{\lambda/\mu})^{n+1}}{\sqrt{\lambda\mu}}\int^t_0\sum^\infty_{k=0}\frac{(n+k+1)I_{n+k+1}((2\sqrt{\lambda\mu})u)e^{-(\lambda+\mu+\xi)u}}{(\sqrt{\lambda/\mu})^{k+1}u}du\\\\
        &+\displaystyle\sum^\infty_{m=0}e^{-(\lambda+\mu+\xi)t}\left[\frac{I_{m+n+i+1}((2\sqrt{\lambda\mu})t)}{(\sqrt{\lambda/\mu})^{m-n+i+1}}-\frac{I_{m+n+i+2}((2\sqrt{\lambda\mu})t)}{(\sqrt{\lambda/\mu})^{m-n+i}}\right]\\\\
        &+\displaystyle (\sqrt{\lambda/\mu})^{n-i}I_{n-i}(2(\sqrt{\lambda\mu})t)e^{-(\lambda+\mu+\xi)t}.
        \label{pn}
    \end{array}
    \end{equation}

Using \eqref{wright}, \eqref{p1} and \eqref{pn} in \eqref{resub} we get the state probabilities for the fractional model. In other words, we obtain the following result.

\begin{teo}
Let $\{X^{\alpha}_t\}_{t\geq 0}$ be the fractional queue model $M/M/1$ with catastrophes of parameter $\alpha\in(0,1]$. Let $\beta_{\lambda, \mu, \xi}(x):=x^{-1}e^{-(\lambda+\mu+\xi)x}$,  and let $M_\alpha(x):=W_{-\alpha,1-\alpha}(-x)$, where $W_{-\alpha,1-\alpha}(-x)$ is the Wright distribution with parameter $\alpha$. Then, 

    \begin{equation}
    \begin{array}{ccl}
            P^\alpha_{i,0}(t)&=&\displaystyle\frac{1}{\mu}\int^\infty_0 \left\{ \frac{M_{\alpha}(yt^{-\alpha})}{t^{\alpha}} \beta_{\lambda, \mu, \xi}(y) \sum ^\infty_{n=i}\frac{(n+1)I_{n+1}((2\sqrt{\lambda\mu})y)}{(\sqrt{\lambda/\mu)}^{n+1}} \right\} dy\\[.6cm]
            
      &  &\displaystyle+\frac{\xi}{\mu}\int^\infty_0 \left\{ \frac{M_{\alpha}(yt^{-\alpha})}{t^{\alpha}}  \int^y_0 \left\{ \beta_{\lambda, \mu, \xi}(u)  \sum^\infty_{n=1}\frac{nI_{n+1}((2\sqrt{\lambda\mu})u)}{(\sqrt{\lambda/\mu})^{n}} \right\}du \right\} dy
        \end{array}
    \end{equation}
        
    and, for $n>0$
    
    \begin{equation}
    \begin{array}{ccl}
        \displaystyle P^\alpha_{i,n}(t)&=&\displaystyle\frac{\xi(\sqrt{\lambda/\mu})^{n+1}}{\sqrt{\lambda\mu}} \displaystyle\int^\infty_0  \frac{M_{\alpha}(yt^{-\alpha})}{t^{\alpha}}\left\{ \int^y_0 \beta_{\lambda, \mu, \xi}(u)\sum^\infty_{k=0}\frac{(n+k+1)I_{n+k+1}((2\sqrt{\lambda\mu})u)}{(\sqrt{\lambda/\mu})^{k+1}}du\right\}  dy\\[.7cm]
    &&    \displaystyle+\int^\infty_0 y \beta_{\lambda, \mu, \xi}(y)  \frac{M_{\alpha}(yt^{-\alpha})}{t^{\alpha}} \sum^\infty_{m=0} \left[\frac{I_{m+n+i+1}((2\sqrt{\lambda\mu})y)}{(\sqrt{\lambda/\mu})^{m-n+i+1}}-\frac{I_{m+n+i+2}((2\sqrt{\lambda\mu})y)}{(\sqrt{\lambda/\mu})^{m-n+i}}\right] dy\\[.7cm]
&&        \displaystyle +\int^\infty_0(\sqrt{\lambda/\mu})^{n-i}I_{n-i}(2(\sqrt{\lambda\mu})y)e^{-(\lambda+\mu+\xi)y} \frac{M_{\alpha}(yt^{-\alpha})}{t^{\alpha}}dy.
        
    \end{array}
    \end{equation}

\end{teo}    

We show in Figure \ref{fig:behaviours} the behavior of the states probabilities for different values of $\alpha$. In such an illustration, we assume that the queue starts with $i=1$, and we consider two cases for $n$ at time $t$; namely, we assume $n = 0$ in Figure \ref{subfig:example_p10} and $n = 1$ in Figure \ref{subfig:example_p11}. We are assuming the arrival rate, the departure rate and the catastrophes rate as $\lambda = 5$, $\mu = 3$ and $\xi = 1$, respectively. We can see that the more we decrease $\alpha$, the slower the convergence of the probability becomes.

\begin{figure}[H]
     \centering
     \begin{subfigure}[a]{0.75\textwidth}
         \centering
         \includegraphics[width=\textwidth]{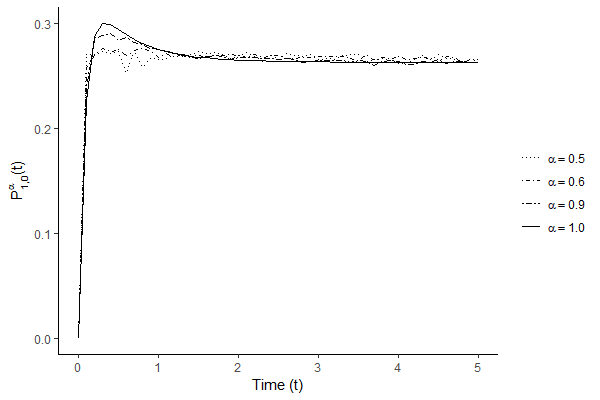}
         \caption{$P^\alpha_{1,0}(t)$ for different values of $\alpha$.}
         \label{subfig:example_p10}
     \end{subfigure}
     
     \begin{subfigure}[a]{0.75\textwidth}
         \centering
         \includegraphics[width=\textwidth]{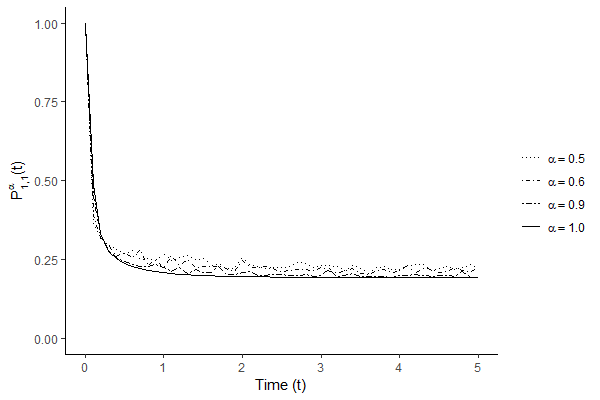}
         \caption{$P^\alpha_{1,1}(t)$ for different values of $\alpha$.}
         \label{subfig:example_p11}
     \end{subfigure}
        \caption{Behavior of the state probability as a function of $t$. Here we consider a queue with catastrophes starting in $i = 1$ with $\lambda = 5$, $\mu = 3$ and $\xi = 1$.}
        \label{fig:behaviours}
\end{figure}

\subsection{Expectation and variance}

In this section we deal with the moments of $X_t^{\alpha}$, and we do it by means of the p.g.f. 
$$G^\alpha(z,t)=\sum^\infty_{n=0}z^iP^\alpha_{i,n}(t),~t>0,~0<\alpha<1.$$ 


We will discuss two ways of using $G^\alpha(z,t)$: on the one hand we deal with the inverse Laplace transform following the ideas of \cite{bayle, fila}; on the other hand, we deal with the manipulation of $G^{\alpha}(z,t)$ through the Mittag-Leffler function like in \cite{kumar}. We point out that both approaches are useful. Let us start with some manipulations of the p.g.f. of $X^{\alpha}_t$.

\begin{teo}
Let $\{X^{\alpha}_t\}_{t\geq 0}$ be a fractional queue model $M/M/1$ with catastrophes of parameter $\alpha\in(0,1]$, and let $G^\alpha(z,t)$ be the p.g.f. Then,
\begin{equation}\label{funcaogeradora}
\begin{split}
G^\alpha(z,t)=z^iE_{\alpha,1}(A(z)t^\alpha)-\mu\left(\frac{1}{z}-1\right)[P^\alpha_{i,0}(t)]\ast & [t^{\alpha-1} E_{\alpha,\alpha}(A(z)t^\alpha)]\\
&+t^\alpha\xi E_{\alpha,\alpha+1}(A(z)t^\alpha),    
\end{split}
\end{equation}
where $A(z)=A_{\lambda,\mu,\xi}(z):=\lambda z-(\lambda+\mu+\xi)+\mu/z$,  
\begin{equation}
    E^\delta_{\beta,\gamma}(w)=\sum^\infty_{r=0}\frac{w^r\Gamma(\delta+r)}{r!\Gamma(\beta r+\gamma)\Gamma(\delta)},~~~~~~~~~~~~~ \gamma, \delta, \beta \in \mathbb{C}, \Re(\beta)>0
\end{equation}
is the 3 parameters Mittag-Leffler function, and $\ast$ is the convolution operator of t.
\label{Gteo}
\end{teo}
\begin{proof}
Remember that $\tilde{G}^\alpha(z,s)=\int_{0}^{\infty}e^{-st}G^{\alpha}(z,t)dt,$ so \eqref{trans1eq} reads
\begin{equation}
    s^\alpha \tilde{G}^\alpha(z,s)-s^{\alpha-1}z^i=\left(z\lambda-(\mu+\lambda +\xi)+\mu/z \right)\tilde{G}^\alpha(z,s)-\mu\left(\frac{1}{z}-1\right) \tilde{P}^\alpha_{i,0}(s)+\frac{\xi z}{s}.
\end{equation}

After some algebraic manipulations and setting $A(z):=\lambda z-(\lambda+\mu+\xi)+\mu/z$, we have
\begin{equation}
    (s^\alpha-A(z)) \tilde{G}^\alpha(z,s)=s^{\alpha-1}z^i-\mu\left(\frac{1}{z}-1\right) \tilde{P}^\alpha_{i,0}(s)+\xi z s^{-1}.
\end{equation}
 
Leaving $s^\alpha-A(z)$ on the right side and using the following inverse Laplace transform, involving the three parameters Mittag-Leffler function,
 
 \begin{equation}
 \label{eq:transformada_mittag}
 \int^\infty_0e^{-st}t^{\gamma-1}E^\delta_{\beta,\gamma}(wt^\beta)=\frac{s^{\beta\delta-\gamma}}{(s^\beta-w)^\delta},
 \end{equation}
 (see \cite[Equation 2.3.24]{mathai} for more details), we conclude that  
 \begin{equation}
\begin{split}
G^\alpha(z,t)=z^iE_{\alpha,1}(A(z)t^\alpha)-\mu\left(\frac{1}{z}-1\right)[P^\alpha_{i,0}(t)]\ast & [t^{\alpha-1} E_{\alpha,\alpha}(A(z)t^\alpha)]\\
&+t^\alpha\xi E_{\alpha,\alpha+1}(A(z)t^\alpha).    
\end{split}
\end{equation}
\end{proof}

\begin{teo}
Let $\{X^{\alpha}_t\}_{t\geq 0}$ be a fractional queue model $M/M/1$ with catastrophes of parameter $\alpha\in(0,1]$, and let $G^\alpha(z,t)$ be the p.g.f. The Laplace transform of $G^\alpha(z,s)$ is given by
    \begin{equation}
        \tilde{G}^\alpha(z,s)=\frac{z^{i+1}s^{\alpha-1}+\xi zs^{-1}-(1-z)\mu \tilde{P}^\alpha_{i,0}}{-\lambda(z-a_1)(z-a_2)}
        \label{glaplace}
    \end{equation}
    where $a_1$ and $a_2$ are zeros of $f(z)=-\lambda z^2+(s^\alpha+\lambda+\mu+\xi)z-\mu$.
\end{teo}
\begin{proof}
    We start the proof as in Theorem \ref{Gteo}, but $z$ is multiplying both sides of the equation 
    \begin{equation}
        z[s^\alpha\tilde{G}^\alpha(z,s)-z^is^{\alpha-1}]=(1-z)[\tilde{G}^\alpha(z,s)(\mu-\lambda z-\frac{\xi z}{(1-z)})-\mu\tilde{P}^\alpha_{i,0}(s)]+\frac{z\xi}{s}.
    \end{equation}

    Leaving $\tilde{G}^\alpha(z,s)$ in the left side, we have
    \begin{equation}
        \tilde{G}^\alpha(z,s)(-\lambda z^2+(s^\alpha+\lambda+\mu+\xi)z-\mu)=z^{i+1}s^{\alpha-1}-(1-z)\mu\tilde{P}^\alpha_{i,0}(s)+\frac{z\xi}{s},
    \end{equation}
    
    that is,
    \begin{equation}
        \tilde{G}^\alpha(z,s)=\frac{z^{i+1}s^{\alpha-1}-(1-z)\mu\tilde{P}^\alpha_{i,0}(s)+z\xi s^{-1}}{-\lambda z^2+(s^\alpha+\lambda+\mu+\xi)z-\mu}.
    \label{denominador}
    \end{equation}

    Let $a_1$ and $a_2$ be the zeros of $f(z):=-\lambda z^2+(s^\alpha+\lambda+\mu+\xi)z-\mu$. The equations of $a_1$ e $a_2$ are given by:
    \begin{equation}
        \begin{cases}
        \displaystyle a_1+a_2=\frac{(s^\alpha+\lambda+\mu+\xi)}\lambda,\\
        \displaystyle a_1a_2=\frac{\mu}{\lambda},\\
        s^\alpha+\xi=-\lambda(1-a_2)(1-a_1).\\
        \end{cases}
        \label{a1a2}
    \end{equation}

  Therefore, 
    \begin{equation}
        \tilde{G}^\alpha(z,s)=\frac{z^{i+1}s^{\alpha-1}+\xi zs^{-1}-(1-z)\mu P^\alpha_{i,0}(s)}{-\lambda(z-a_1)(z-a_2)}.
    \end{equation}
\end{proof}

\begin{teo}
Let $\{X^{\alpha}_t\}_{t\geq 0}$ be a fractional queue model $M/M/1$ with catastrophes of parameter $\alpha\in(0,1]$. The expectation of $X^\alpha_t$ is given by
\begin{equation}\label{expec}
    \mathbb{E}[X^\alpha_t]= iE^1_{\alpha,1}(-\xi t^\alpha)+\mu [P^\alpha_{i,0}(t)]\ast[ t^{\alpha-1}E^1_{\alpha,\alpha}(-\xi t^\alpha)]+\frac{(\lambda-\mu)}{\xi}(1-E^1_{\alpha,1}(-\xi t^\alpha)).
\end{equation}
 
\label{teomean}
\end{teo}
\begin{proof}
In order to prove \eqref{expec} we apply the derivative of \eqref{glaplace}, so
\begin{equation}
\begin{array}{rl}
    \displaystyle\frac{\partial \tilde{G}^\alpha(z,s)}{\partial z}=&\displaystyle\frac{[(i+1)z^is^{\alpha-1}+\xi s^{-1}+\mu\tilde{P}^\alpha_{i,0}(s)](-\lambda(z-a_1)(z-a_2))}{\lambda^2(z-a_1)^2(z-a_2)^2}\\[.4cm]
    &
    \displaystyle +\frac{(z^{i+1}s^{\alpha-1}+\xi zs^{-1}-(1-z)\mu\tilde{P}^\alpha_{i,0}(s))(\lambda[z-a_1+z-a_2] )}{\lambda^2(z-a_1)^2(z-a_2)^2}.
    \end{array}
    \label{derivada1g}
\end{equation}
Using $z=1$ and \eqref{a1a2}, we have
\begin{equation}
\begin{array}{rl}
\displaystyle\frac{\partial \tilde{G}^\alpha(1,s)}{\partial z}=&\displaystyle\frac{((i+1)s^{\alpha-1}+\xi s^{-1}+\mu\tilde{P}^\alpha_{i,0}(s)}{-\lambda(1-a_1)(1-a_2)}+\frac{(s^{\alpha-1}+\xi s^{-1})\lambda[1-a_1+1-a_2] )}{\lambda^2(1-a_1)^2(1-a_2)^2}\\[.4cm]
    =&\displaystyle\frac{(i+1)s^{\alpha-1}+\xi s^{-1}+\mu \tilde{P}^\alpha_{i,0}(s)}{\xi +s^\alpha}
+\displaystyle\frac{(s^{\alpha-1}+\xi s^{-1})\lambda[2-(a_1+a_2)] )}{(\xi+s^\alpha)^2}\\[.4cm]
   =&\displaystyle\frac{(i+1)s^{\alpha-1}+\xi s^{-1}+\mu \tilde{P}^\alpha_{i,0}(s)}{\xi+s^\alpha}
+\frac{(s^{\alpha-1}+\xi s^{-1})(\lambda-\mu-\xi-s^\alpha )}{(\xi+s^\alpha)^2}\\[.4cm]
   =&\displaystyle\frac{is^{\alpha-1}+\mu \tilde{P}^\alpha_{i,0}(s)}{\xi+s^\alpha}
+\frac{(s^{\alpha-1}+\xi s^{-1})(\lambda-\mu-\xi-s^\alpha +\xi +s^\alpha)}{(\xi+s^\alpha)^2}\\[.4cm]
   =&\displaystyle\frac{is^{\alpha-1}+\mu \tilde{P}^\alpha_{i,0}(s)}{\xi+s^\alpha}
+\frac{s^{-1}(s^{\alpha}+\xi)(\lambda-\mu)}{(\xi+s^\alpha)^2}\\[.4cm]
   =&\displaystyle\frac{is^{\alpha-1}+\mu \tilde{P}^\alpha_{i,0}(s)+s^{-1}(\lambda-\mu)}{\xi+s^\alpha}.
   \end{array}
\end{equation}
Finally, using \eqref{eq:transformada_mittag}, we obtain 
\begin{equation}
    \mathbb{E}[X^\alpha_t]= iE^1_{\alpha,1}(-\xi t^\alpha)+\mu P^\alpha_{i,0}(t)\ast t^{\alpha-1}E^1_{\alpha,\alpha}(-\xi t^\alpha)+(\lambda-\mu)t^\alpha E^1_{\alpha,\alpha+1}(-\xi t^\alpha)
    \label{premedia}
\end{equation}
and we conclude the proof of the theorem by \cite[page 82, Theorem 2.2.1]{mathai}, where
\begin{equation}
    \label{eq:221mathai}
    E^1_{\alpha,\beta}(z) = z E^1_{\alpha,\alpha + \beta}(z) + \frac{1}{\Gamma (\beta)}.
\end{equation}
\end{proof}

\begin{no}
By considering $\xi=0$ we recover the result obtained by \cite[equation 2.27]{fila}. For this, we need to use $E^\delta_{\alpha, \beta}(0)=\Gamma(\beta)^{-1}$ in \eqref{premedia}. More precisely, we recover 
\begin{equation}
    \mathbb{E}[X^\alpha_t]= i+\mu J^\alpha P^\alpha_{i,0}(t)+\frac{(\lambda-\mu)t^\alpha}{\Gamma (\alpha +1)}
\end{equation}

where $J^\alpha f(t)$ is the {Riemann-Liouville} fractional integral \cite{kilbas}
\begin{equation}
    J^\alpha f(t)=\frac{1}{\Gamma(\alpha)}\int^t_0(y-t)^{\alpha-1}f(y)dy.
\end{equation}
\end{no}
\begin{no}
By letting $\alpha = 1$, we get the classical result for the $M/M/1$ queue with catastrophes obtained by \cite[Equation 2.25]{kumar}. Note that $E^1_{1,1}(w) = e^w$.
\end{no}

\begin{no}
We emphasize that an alternative way to prove Theorem \ref{teomean} is a suitable application of Theorem \ref{Gteo}, and the fact that
$$\frac{\partial G^\alpha(1,t)}{\partial z}=\mathbb{E}[X^\alpha_t].$$
For this, we differentiate \eqref{funcaogeradora}, using 
$$\frac{dE^\delta_{\alpha,\beta}(t)}{d t}=\frac{\Gamma(\delta + 1)}{\Gamma(\delta)}E^{\delta+1}_{\alpha,\beta + \alpha}(t),$$ 
see \cite[Equation 2.2.1]{shukla}, so 
\begin{equation}
\begin{array}{rl}
    \displaystyle \frac{\partial G^\alpha(z,t)}{\partial z}=&\displaystyle i z^{i-1}E^1_{\alpha,1}(A(z)t^\alpha)+z^it^\alpha\left(\lambda - \frac{\mu}{z^2}\right)E^2_{\alpha, \alpha+1}(A(z)t^\alpha)\\[0.4cm]
    &\displaystyle +\frac{\mu}{z^2}[P^\alpha_{i,0}(t)]\ast [t^{\alpha-1}E^1_{\alpha, \alpha}(A(z)t^\alpha) ] \\[0.4cm] 
    &+\displaystyle\mu\left(\frac{1}{z}-1\right)\left(\lambda - \frac{\mu}{z^2}\right)[P^\alpha_{i,0}(t)]\ast [t^{2\alpha-1}E^2_{\alpha, 2\alpha}(A(z)t^\alpha) ]\\[0.4cm]
    &\displaystyle + t^{2\alpha}\xi\left(\lambda - \frac{\mu}{z^2}\right)E^2_{\alpha, 2\alpha + 1}(A(z)t^\alpha).
    \end{array}
\end{equation}
Taking $z=1$, we have
\begin{equation}
\begin{array}{rl}
     \displaystyle\frac{\partial G^\alpha(1,t)}{\partial z}=&iE^1_{\alpha,1}(-\xi t^\alpha)+t^\alpha\left(\lambda - \mu\right)E^2_{\alpha, \alpha+1}(-\xi t^\alpha)\\[0.4cm]
     &+\displaystyle\mu[P^\alpha_{i,0}(t)]\ast [t^{\alpha-1}E^1_{\alpha, \alpha}(-\xi t^\alpha) ]+t^{2\alpha}\xi(\lambda- \mu)E^2_{\alpha, 2\alpha + 1}(-\xi t^\alpha),
     \end{array}
     \end{equation}
and, since  
\begin{equation}
\label{eq:238mathai}
 E^\delta_{\alpha,\beta-\alpha}(z)-E^{\delta-1}_{\alpha,\beta-\alpha}(z)=zE^\delta_{\alpha,\beta}(z),
 \end{equation}
see \cite[Equation 2.3.8]{mathai}, we get the desired result.
%
\end{no}


\begin{teo}
Let $\{X^{\alpha}_t\}_{t\geq 0}$ be the fractional queue model $M/M/1$ with catastrophes of parameter $\alpha\in(0,1]$. The variance of $X^\alpha_t$ is given by
\begin{equation}
     \begin{split}
    Var(X^\alpha_t)=&\;i^2E^1_{\alpha,\alpha+1}(-\xi t^\alpha)+2it^\alpha\left(\lambda-\mu\right)E^2_{\alpha, \alpha+1}(-\xi t^\alpha) 
    -\mu[P^\alpha_{i,0}(t)\ast t^{\alpha-1}E^1_{\alpha,\alpha}(-\xi t^\alpha)]\\[.4cm]
    &    \displaystyle +2\mu(\lambda-\mu)[P^\alpha_{i,0}(t)\ast t^{\alpha-1}E^2_{\alpha,2\alpha}(-\xi t^\alpha)] +(\mu + \lambda)t^\alpha E^1_{\alpha, \alpha+1}(-\xi  t^\alpha)\\[.4cm]
    &\displaystyle+2t^{2\alpha}(\lambda-\mu)^2E^2_{\alpha, 2\alpha+1} -[iE^1_{\alpha,1}(-\xi t^\alpha)+\mu P^\alpha_{i,0}(t)\ast t^{\alpha-1}E^1_{\alpha,\alpha}(-\xi t^\alpha)\\[.4cm]
    &\displaystyle + \frac{(\lambda-\mu)}{\xi}(1-E^1_{\alpha,1}(-\xi t^\alpha))]^2.
    \end{split}
    \end{equation}

\end{teo}

\begin{proof}
Since $$\frac{\partial^2 G^\alpha(1,t)}{\partial z^2}=\mathbb{E}[(X^\alpha_t)^2]-\mathbb{E}[X^\alpha_t],$$
we can find $E[(X^\alpha_t)^2]$ by noting that
\begin{equation}
    \begin{split}
    \displaystyle\frac{\partial^2 G^\alpha(z,t)}{ \partial z^2} = &\displaystyle
    \;i(i-1)z^{i-2}E^1_{\alpha,\alpha+1}(A(z)t^\alpha)+iz^{i-1}t^\alpha
    \left(\lambda -\frac{\mu}{z^2}\right)E^2_{\alpha, \alpha+1}(A(z)t^\alpha)\\
    &+t^\alpha(i\lambda z^{i-1}-(i-2)\mu  z^{i-3})E^2_{\alpha, \alpha+1}(A(z)t^\alpha)
    \displaystyle+2t^{2\alpha}z^i\left(\lambda-\frac{\mu}{z^2}\right)^2E^3_{\alpha,2\alpha+1}(A(z)t^\alpha)\\
    &-  \displaystyle\frac{2\mu}{z^3} [P^\alpha_{i,0}(t)\ast t^{\alpha-1} E^1_{\alpha,\alpha}(A(z)t^\alpha)]+\frac{\mu}{z^2}\left(\lambda-\frac{\mu}{z^2}\right)[P^\alpha_{i,0}(t)\ast t^{2\alpha-1} E^2_{\alpha,2\alpha}(A(z)t^\alpha)]\\
    &-\displaystyle \mu\left[-\frac{\lambda}{z}+\frac{3\mu}{z^4}-\frac{2\mu}{z^3}\right][P^\alpha_{i,0}(t)\ast t^{2\alpha-1} E^2_{\alpha,2\alpha}(A(z)t^\alpha)]\\
    &\displaystyle-2\mu\left(\frac{1}{z}-1\right)\left(\lambda+\frac{\mu}{z^2}\right)^2[P^\alpha_{i,0}(t)\ast t^{3\alpha-1} E^3_{\alpha,3\alpha}(A(z)t^\alpha)]\\
    &    \displaystyle+2t^{2\alpha}\xi\frac{\mu}{z^3}E^2_{\alpha,2\alpha+1}(A(z)t^\alpha)\displaystyle+2\xi t^{3\alpha}\left(\lambda-\frac{\mu}{z^2}\right)^2E^3_{\alpha,3\alpha+1}(A(z)t^\alpha).
    \end{split}
\end{equation}
Thus, by taking $z=1$, we find
\begin{equation}
\begin{split}
   \displaystyle \frac{\partial^2 G^\alpha(1,t)}{\partial z^2}=&\;i(i-1)E^1_{\alpha,1}(-\xi t^\alpha)+it^\alpha(\lambda-\mu)E^2_{\alpha, \alpha+1}(-\xi t^\alpha)\\
   &\displaystyle+t^\alpha(i\lambda -(i-2)\mu )E^2_{\alpha, \alpha+1}(-\xi t^\alpha) +2t^{2\alpha}\left(\lambda-\mu\right)^2E^3_{\alpha,2\alpha+1}(-\xi t^\alpha)\\
    & -2\mu [P^\alpha_{i,0}(t)\ast t^{\alpha-1} E^1_{\alpha,\alpha}(-\xi t^\alpha)]+\mu\left(\lambda-\mu\right)[P^\alpha_{i,0}(t)\ast t^{2\alpha-1} E^2_{\alpha,2\alpha}(-\xi t^\alpha)]\\
    &\displaystyle+\mu(\lambda-\mu)[P^\alpha_{i,0}(t)\ast t^{2\alpha-1} E^2_{\alpha,2\alpha}(-\xi t^\alpha)]\\
    &+2t^{2\alpha}\xi\mu E^2_{\alpha,2\alpha+1}(-\xi t^\alpha)+2\xi t^{3\alpha}\left(\lambda-\mu\right)^2E^3_{\alpha,3\alpha+1}(-\xi t^\alpha).
    \end{split}
    \end{equation}
By applying Equation \eqref{eq:238mathai} and after some algebraic manipulations we get
\begin{equation}\label{expec2}
\begin{split}
    \displaystyle\mathbb{E}[(X^\alpha_t)^2]-\mathbb{E}[X^\alpha_t]=&\;i(i-1)E^1_{\alpha,1}(-\xi t^\alpha) + 2it^\alpha\left(\lambda-\mu\right)E^2_{\alpha, \alpha+1}(-\xi t^\alpha)\\
    &-2\mu[P^\alpha_{i,0}(t)\ast t^{\alpha-1}E^1_{\alpha,\alpha}(-\xi t^\alpha)]\\
    &+ 2\mu(\lambda-\mu)[P^\alpha_{i,0}(t)\ast t^{2\alpha-1}E^2_{\alpha,2\alpha}(-\xi t^\alpha)]\\
    &+2\mu t^\alpha E^1_{\alpha, \alpha+1}(-\xi  t^\alpha)+2t^{2\alpha}(\lambda-\mu)^2E^2_{\alpha, 2\alpha+1}(-\xi t^\alpha).
    \end{split}
    \end{equation}
    Finally, we obtain the second moment for $X_{t}^{\alpha}$ by \eqref{premedia} and \eqref{expec2}; namely
    \begin{equation}
    \begin{array}{rcl}
    E[(X^\alpha_t)^2]&=&i^2E^1_{\alpha,\alpha+1}(-\xi t^\alpha)+2it^\alpha\left(\lambda-\mu\right)E^2_{\alpha, \alpha+1}(-\xi t^\alpha)-\mu[P^\alpha_{i,0}(t)\ast t^{\alpha-1}E^1_{\alpha,\alpha}(-\xi t^\alpha)]\\[.4cm]
&&    \displaystyle +2\mu(\lambda-\mu)[P^\alpha_{i,0}(t)\ast t^{2\alpha-1}E^2_{\alpha,2\alpha}(-\xi t^\alpha)]\\[.4cm]
&&+ (\lambda + \mu) t^\alpha E^1_{\alpha, \alpha+1}(-\xi  t^\alpha)+2t^{2\alpha}(\lambda-\mu)^2E^2_{\alpha, 2\alpha+1}(-\xi t^\alpha).
    \end{array}
    \end{equation}
 Since $Var(X^\alpha_t)=E((X^\alpha_t)^2)-E(X^\alpha_t)^2 $ the proof is complete.
\end{proof}

Assuming the parameters of Figure \ref{fig:behaviours}, we show in Figure \ref{fig:esp_var_behaviours} the behavior of the expected value and the variance during time.

\begin{figure}[H]
     \centering
     \begin{subfigure}[a]{0.75\textwidth}
         \centering
         \includegraphics[width=\textwidth]{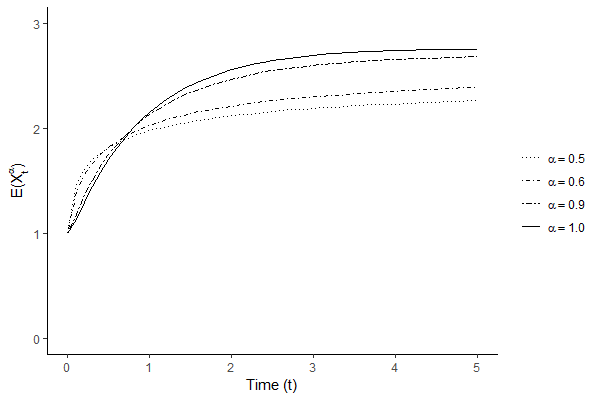}
         \caption{$\mathbb{E}[X^\alpha_t]$ for different values of $\alpha$.}
         \label{subfig:example_esp}
     \end{subfigure}
     \begin{subfigure}[a]{0.75\textwidth}
         \centering
         \includegraphics[width=\textwidth]{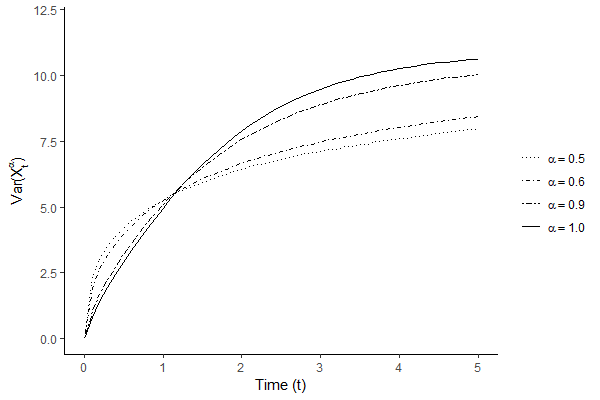}
         \caption{$Var(X^\alpha_t)$ for different values of $\alpha$.}
         \label{subfig:example_var}
     \end{subfigure}
        \caption{Behavior of the mean and the variance along time of a fractional queue with catastrophes starting in $i = 1$ with $\lambda = 5$, $\mu = 3$ and $\xi = 1$.}\label{fig:esp_var_behaviours}
\end{figure}

    \section{Estimation of parameters}

In this section we discuss about the estimation of parameters for the model. In order to do it we start with some remarks about the waiting times for the fractional $M/M/1$ queue with catastrophes. 
Let $S_k$ be the waiting time until something happens in a classical queue, provided the queue is not empty. Then, $S_k$ follows an exponential distribution with parameter $\lambda + \mu + \xi$. On the other hand, let $S^\alpha_k$ be the waiting time until something happens in a fractional queue when the queue is not empty. Then, thanks to the subordination relationship \eqref{resub}, we have that it holds
\begin{equation}
    P(S^\alpha_k>t)=\int^\infty_0exp[-(\lambda+\mu+\xi)s]f_{\alpha}(s,t)ds.
\end{equation}

By applying the Laplace transform, we get
$$\int^\infty_0e^{-zt}P(S^\alpha_k>t)dt =\int^\infty_0e^{-(\lambda+\mu+\xi)s}z^{\alpha - 1}e^{sz^\alpha}ds = \frac{z^{\alpha - 1}}{\lambda + \mu + \xi + z^\alpha}.$$
Applying Equation \eqref{eq:transformada_mittag} and the property of the Mittag-Leffler functions, described in Equation \eqref{eq:221mathai}, we find
$$P(S^\alpha_k>t) = E^1_{\alpha, 1}(-(\lambda + \mu + \xi)t^\alpha) = 1 - (\lambda + \mu + \xi)t^\alpha E^1_{\alpha,\alpha + 1}(-(\lambda + \mu + \xi)t^\alpha).$$
Finally, since $f_{S^\alpha}(t) = -dP(S^\alpha_k>t)/dt$ and $d[t^{\beta - 1}E^1_{\beta,\gamma}(\omega t^\alpha)]/dt = t^{\beta - 2}E^1_{\beta,\gamma - 1}(\omega t^\alpha)$ (see \cite[Equation 2.2.3]{mathai}), we get
\begin{equation}
   f_{S^\alpha}(t) = (\lambda+\mu+\xi)t^{\alpha-1}E_{\alpha,\alpha}(-(\lambda+\mu+\xi)t^\alpha).
    \label{distmitt}
\end{equation}

From \eqref{distmitt} we obtain that $S^\alpha_k$ follows a Mittag-Leffler distribution, therefore $S^\alpha_k$ can be represented as 
\begin{equation}
    S^\alpha_k=\mathcal{E}^{\frac{1}{\alpha}}T_\alpha,
    \label{set}
\end{equation}
where $\mathcal{E}$ follows an exponential distribution of parameter $\theta=\lambda+\mu+\xi$ and $T_\alpha$ follows a one-sided $\alpha^+$-stable law. If we let $ln(S^\alpha_k)=\frac{1}{\alpha}ln(\mathcal{E})-ln(T_\alpha)$, then we can write from \cite{cahoy2010} 

\begin{equation}
    \mu_{\ln{S^\alpha_k}}= -\frac{ln(\theta)}{\alpha}-\gamma
    \label{esplns}
\end{equation}
and

\begin{equation}
    \sigma^2_{\ln{S^\alpha_k}}=\pi^2\left(\frac{1}{3\alpha^2}-\frac{1}{6} \right),
    \label{varlns}
\end{equation}
where $\gamma$ is the Euler-Mascheroni constant described by $\gamma = 0.57721\dots$ From the previous results, we can deduce the moment estimator of $\alpha$ and $\theta$; namely,
\begin{equation}
    \hat{\alpha}=\frac{\pi}{\sqrt{3\left(\hat{\sigma}^2_{\ln{S^\alpha_k}}+\frac{\pi^2}{6}\right)}},
    \label{mmalpha}
\end{equation}
and

\begin{equation}
    \hat{\theta}=exp(-\hat{\alpha}(\hat{\mu}_{\ln{S^\alpha_k}}+\gamma))
    \label{mmtheta}
\end{equation}
\smallskip
where $\displaystyle{\hat{\mu}_{{\ln{S^{\alpha}_k}}}=\sum^n_{j=1}\frac{\ln(S^{\alpha}_j)}{n}}$\text{ and } $\displaystyle{\hat{\sigma}^{2}}_{\ln{S^\alpha_k}}=\sum^n_{j=1}\frac{(\ln(S^{\alpha}_j)-\hat{\mu}_{\ln(S^{\alpha}_k)})^2}{n}$. The asymptotic normality and confidence interval for $\alpha$ are, respectively, given by
\begin{equation}
    \sqrt{n}(\hat{\alpha}-\alpha)\sim  N\left(~0,~\frac{\alpha^2(32-20\alpha^2-\alpha^4)}{40}\right)
\end{equation}
and
\begin{equation}
    \hat{\alpha}\pm z_{\epsilon/2}\sqrt{\frac{\alpha^2(32-20\alpha^2-\alpha^4)}{40n}},
\end{equation}
as proved in \cite{cahoy2010}. 

We are interested in estimating the parameters $\lambda,\mu,$ and $\xi$, and we shall do it by applying the arguments developed in \cite{fila}. Although the following results are obtained by adapting arguments from \cite{fila}, we include them here for the sake of clarity. By using the proportion of occurrences and the asymptotic properties of each event, like in \cite{fila}, we get a point and interval estimation.
In other words, let $n_a$, $n_s$ and $n_c$ be the number of costumers that arrive to the system, the number of costumers that leave the system and the number of catastrophes, respectively. Let $n=n_a+n_s+n_c$. We can write the arrival proportion and its estimator as $\lambda/\theta=p_1$ and $\hat{p_1}=n_a/n$. We do the same for the proportion of costumers leaving the system: $\mu/\theta=p_2$ and $\hat{p_2}=n_s/n$, and for the catastrophes: $\xi/\theta=p_3$ and  $\hat{p_3}=n_c/n$. Now we shall show the asymptotic normality for the parameters $\lambda,~\mu$, and $\xi$, respectively.

\begin{teo}
For $n\rightarrow\infty$
\begin{align}
    \sqrt{n}(\hat{\lambda}-\lambda)\sim N\left(0,\theta^2p_1(1-p_1)+p^2_1\sigma^2_\theta\right)\\
    \sqrt{n}(\hat{\mu}-\mu)\sim N\left(0,\theta^2p_2(1-p_2)+p^2_2\sigma^2_\theta\right)\\
    \sqrt{n}(\hat{\xi}-\xi)\sim N\left(0,\theta^2p_3(1-p_3)+p^2_3\sigma^2_\theta\right)
\end{align}

where $\hat{p_1}=n_a/n$, $p_1=\lambda/\theta$, $\hat{p_2}=n_s/n$, $p_2=\mu/\theta$, $\hat{p_3}=n_c/n$, $p_3=\xi/\theta$ and
\begin{equation}
    \sigma^2_\theta=\frac{\theta^2[20\pi^4(2-\alpha^2)-3\pi^2(\alpha^4+20\alpha^2-32)(ln(\theta))^2-720\alpha^3(ln(\theta))\zeta(3)]}{120\pi^2},
    \label{vartheta}
\end{equation}
where $\zeta(3)$ is the Riemann-zeta function evaluated in 3.
\end{teo}
\begin{proof}
The proof follows the same steps of the Theorem 3.1 of \cite{fila}, with the difference that here we adapt the number of parameters for the model. In other words, we consider a Multinomial$(1,p_1,p_2,p_3)$ and we use the asymptotic property of the parameters as follows
\begin{equation}
    \sqrt{n}\left(\begin{array}{cc}
         \hat{p_1}-p_1  \\
          \hat{\theta}-\theta
    \end{array}\right)
    \xrightarrow{d} N(0,\Sigma),
\end{equation}
for $n\xrightarrow{}\infty$, where the covariance matrix $\Sigma$ is given by 
\begin{equation}
    \Sigma=\left(\begin{array}{cc}
        p_1(1-p_1) & 0 \\
        0 & \sigma^2_\theta
    \end{array}\right),
\end{equation}
and $\sigma^2_\theta$ is obtained in \cite{cahoy2010}.
Using the central limit theorem we get
\begin{equation}
    \sqrt{n}(h(\hat{\omega}_n)-h(\omega))\sim N(~0,\Dot{h}(\omega)^T\Sigma\Dot{h}(\omega)),
\end{equation}
where $\hat{\omega_n}=(\hat{p_1},\hat{\theta})^T$, $h$ is a mapping from $\mathbb{R}^2\xrightarrow{}\mathbb{R}$, $\Dot{h}$ is the gradient of $h$ and it is continuous in a neighborhood of $\omega\in\mathbb{R}^2$. in this case, $ h(p_1,\theta)=p_1\theta$ and $\Dot{h}(p_1,\theta)=(\theta,p_1)$. The proof follows in the exact same way for $p_2$ and $p_3$. Therefore the proof is complete.
\end{proof}

Once we prove the asymptotic properties above, we can write the confidence intervals (1-$\epsilon$)100\% for $\lambda,~\mu~e~\xi$, respectively 
\begin{equation}
   IC[\hat{\lambda}]= \hat{\lambda}\pm z_{\epsilon/2}\hat{\sigma_\lambda},~~IC[\hat{\mu}]= \hat{\mu}\pm z_{\epsilon/2}\hat{\sigma_\mu}~~\text{ and }~~IC[\hat{\xi}]= \hat{\xi}\pm z_{\epsilon/2}\hat{\sigma_\xi},
\end{equation}
where
\begin{equation}
    \hat{\sigma_\lambda}= \sqrt{\frac{\hat{\theta}\hat{p_1}(1-\hat{p_1})+\hat{p_1}^2\hat{\sigma^2_\theta}}{n}},~~ \hat{\sigma_\mu}= \sqrt{\frac{\hat{\theta}\hat{p_2}(1-\hat{p_2})+\hat{p_2}^2\hat{\sigma^2_\theta}}{n}}~~\text{ and }~~ \hat{\sigma_\xi}= \sqrt{\frac{\hat{\theta}\hat{p_3}(1-\hat{p_3})+\hat{p_3}^2\hat{\sigma^2_\theta}}{n}}.
\end{equation}


Since the waiting times are given by a Mittag-Leffler distribution and \eqref{set} holds, we can simulate values for the validation of the proposed estimators and intervals. Taking any $k\geq0$ for starting the process, we generate the values with the following algorithm, which is a modification of the one proposed by \cite{gillespie1976}: \\[.5cm]

\begin{algorithm}[H]
\begin{enumerate}
        \item[1)] Starts with $X^\alpha_0=i$ and $t = 0$.
        \item[2)] If $X^\alpha_t = k \neq0$:
        \begin{enumerate}
            \item[i)] Generate $\mathcal{E}\sim Exp(\lambda + \mu + \xi)$. 
            \item[ii)] Generate $T^\alpha$ from a one-sided $\alpha^+$-stable distribution. 
            \item[iii)]  Calculate $S^\alpha_k = \mathcal{E}^{\frac{1}{ \alpha}}T^\alpha$ and $t = t + S^\alpha_k$.
            \item[iv)]  Generate U uniformly distributed in (0,1).
            \item[v)]  If $0\leq U < \frac{\lambda}{\lambda+\mu+\xi}$, take $X^\alpha_t=k+1$.
            \item[vi)]  If $\frac{\lambda}{\lambda+\mu+\xi}\leq U < \frac{\lambda + \mu}{\lambda+\mu+\xi}$, we take $X^\alpha_t = k - 1$.
            \item[vii)]  Else we take $X^\alpha_t=0$.
        \end{enumerate}
       
       \item[3)] if $X^\alpha_t=0$:
       
        \begin{enumerate}
           \item[i)] Generate $\mathcal{E}_0\sim Exp(\lambda)$. 
            \item[ii)] Generate $T^\alpha$ from a one-sided $\alpha^+$-stable distribution. 
            \item[iii)]  Calculate $S^\alpha_k = \mathcal{E}_0^{\frac{1}{ \alpha}}T^\alpha$ and calculate $t = t + S^\alpha_k$.
            \item[iv)]  Take $X^\alpha_t = 1$;
        \end{enumerate}
        \item[4)]Repeat until the number of desire iteration.
        \end{enumerate}
        \caption{Fractional queue with catastrophe simulation.}
    \end{algorithm}
    
    \begin{no}
We assume $X^\alpha_t = k$ in Algorithm 1 to simplify notation, we point out that  $X^\alpha_t$ changes its value for every iteration.
    \end{no}

Assuming again a fractional queue starting in $i = 1$ and $\lambda = 5$, $\mu = 3$ and $\xi = 1$, we simulated its behavior during time for different values of $\alpha$, see Figure \ref{fig:esp_var_behaviours2} for more details.  
\begin{figure}[H]
     \centering
     \begin{subfigure}[a]{0.49\textwidth}
         \centering
         \includegraphics[width=\textwidth]{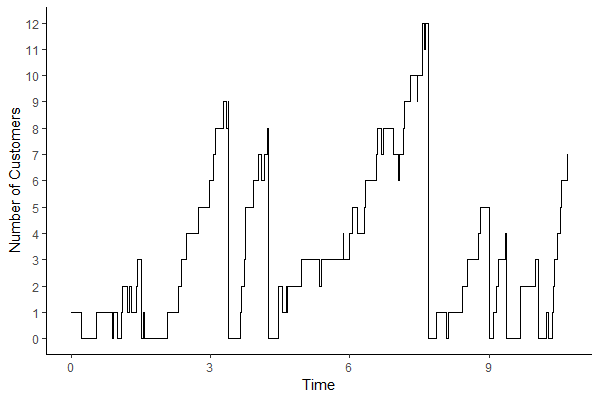}
         \caption{Simulation for $\alpha = 1$.}
         \label{subfig:example_esp}
     \end{subfigure}
     \begin{subfigure}[a]{0.49\textwidth}
         \centering
         \includegraphics[width=\textwidth]{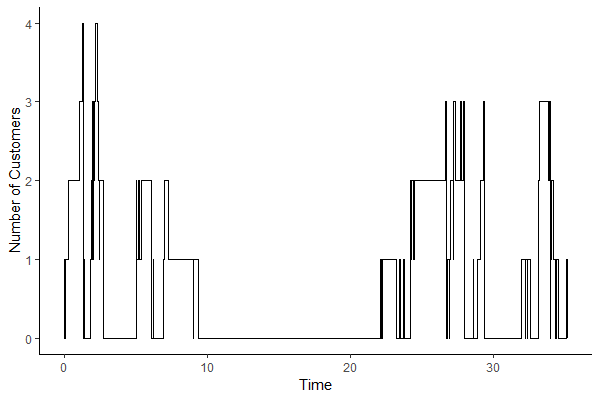}
         \caption{Simulation for $\alpha = 0.9$.}
         \label{subfig:example_var}
     \end{subfigure}
     \begin{subfigure}[a]{0.49\textwidth}
         \centering
         \includegraphics[width=\textwidth]{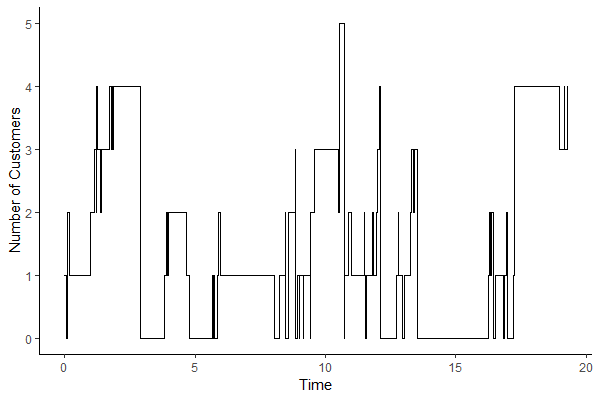}
         \caption{Simulation for $\alpha = 0.8$.}
         \label{subfig:example_var}
     \end{subfigure}
     \begin{subfigure}[a]{0.49\textwidth}
         \centering
         \includegraphics[width=\textwidth]{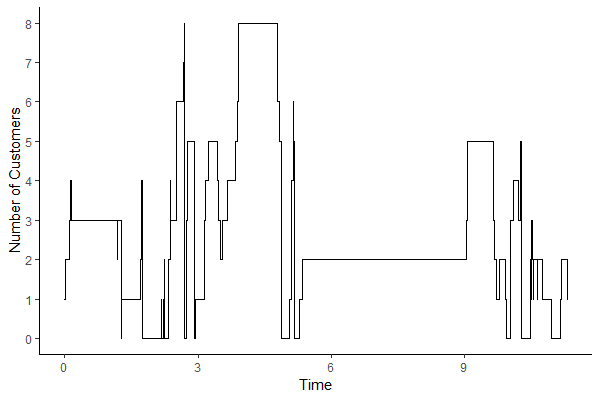}
         \caption{Simulation for $\alpha = 0.7$.}
         \label{subfig:example_var}
     \end{subfigure}
        \caption{Simulation of the fractional queue with catastrophes starting in $i = 1$ with $\lambda = 5$, $\mu = 3$ and $\xi = 1$.}
        \label{fig:esp_var_behaviours2}
\end{figure}

\begin{table}[!htbp] \centering 
  \caption{Estimators and confidential intervals tests} 
  \label{} 
\begin{tabular}{@{\extracolsep{5pt}} lccccccccc} 
\\[-1.8ex]\hline 
\hline \\[-1.8ex]
  & \multicolumn{3}{c}{$n = 10^2$}& \multicolumn{3}{c}{$n = 10^3$} &\multicolumn{3}{c}{$n = 10^4$} \\ 
 & \% Bias & CV & CP & \% Bias & CV & CP & \% Bias & CV & CP \\ 
  \hline 
$\alpha = 0.9$ & $0.656$ & $5.934$ & $0.948$ & $0.022$ & $1.945$ & $0.949$ & $0.021$ & $0.615$ & $0.946$ \\ 
$\lambda = 4$ & $2.877$ & $17.499$ & $0.943$ & $0.374$ & $5.718$ & $0.944$ & $0.077$ & $1.722$ & $0.957$ \\ 
$\mu = 2$ & $3.667$ & $22.269$ & $0.941$ & $0.166$ & $7.031$ & $0.945$ & $0.130$ & $2.088$ & $0.958$ \\ 
$\xi = 1$ & $2.991$ & $31.028$ & $0.919$ & $0.155$ & $9.021$ & $0.954$ & $0.181$ & $2.873$ & $0.951$ \\ 
\hline
$\alpha = 0.5$ & $2.026$ & $8.043$ & $0.955$ & $0.255$ & $2.508$ & $0.960$ & $0.030$ & $0.808$ & $0.953$ \\ 
$\lambda = 1$ & $6.940$ & $41.243$ & $0.932$ & $0.841$ & $12.182$ & $0.952$ & $0.030$ & $3.792$ & $0.963$ \\ 
$\mu = 3$ & $7.540$ & $30.431$ & $0.948$ & $1.025$ & $9.227$ & $0.962$ & $0.141$ & $2.876$ & $0.959$ \\ 
$\xi = 6$ & $8.043$ & $28.494$ & $0.952$ & $1.384$ & $7.903$ & $0.961$ & $0.036$ & $2.572$ & $0.955$ \\ 
\hline
$\alpha = 0.1$ & $1.315$ & $9.007$ & $0.943$ & $0.025$ & $2.817$ & $0.945$ & $0.018$ & $0.943$ & $0.934$ \\ 
$\lambda = 7$ & $8.273$ & $31.447$ & $0.937$ & $0.502$ & $9.004$ & $0.941$ & $0.058$ & $3.034$ & $0.942$ \\ 
$\mu = 0.9$ & $8.990$ & $46.862$ & $0.927$ & $0.604$ & $13.922$ & $0.946$ & $0.219$ & $4.479$ & $0.939$ \\ 
$\xi = 3$& $8.149$ & $34.846$ & $0.941$ & $0.231$ & $10.125$ & $0.946$ & $0.069$ & $3.401$ & $0.936$ \\ 
\hline
$\alpha = 1$& $0.733$ & $5.155$ & $0.944$ & $0.030$ & $1.641$ & $0.953$ & $0.015$ & $0.535$ & $0.951$ \\ 
$\lambda = 2$ & $1.723$ & $18.766$ & $0.918$ & $0.307$ & $5.760$ & $0.950$ & $0.065$ & $1.775$ & $0.952$ \\ 
$\mu = 2$& $1.074$ & $18.041$ & $0.933$ & $0.237$ & $5.770$ & $0.946$ & $0.041$ & $1.859$ & $0.934$ \\ 
$\xi = 2$& $1.258$ & $17.543$ & $0.937$ & $0.116$ & $5.600$ & $0.958$ & $0.009$ & $1.752$ & $0.959$ \\ 
\hline \\[-1.8ex] 
\end{tabular}
\label{bias}
\end{table}

\newpage


With the simulated times, we can test the estimated value of the parameters. To test the estimators we use the percent bias (\% Bias) and coefficient of variation (CV), for which values close to zero indicate that our estimated parameters are close to their real value. We use the coverage probabilities (CP) to check the confidence interval of 95\%, for which values close to 0.95 means that our proposal interval is a good confidence interval.
We test four cases, and for each case we generated $1000$ samples with a sample size equals $10^2$, $10^3$ and $10^4$. We can see the test results in Table \ref{bias} and we can see also that the more we increase the sample size, the better the estimations are. In fact, the sample size of $10^3$ brings a good approach for the examples, since the \% Bias and CV are relatively short and CP is close to $0.95$.

\section{Conclusion}
In this work we complement the analysis of \cite{catastrofef} by appealing to the approach proposed by \cite{fila}. In \cite{catastrofef} the authors focus their attention on the transient behavior of the model determining the transient distribution, the distribution of the busy period and the probability distribution of the time of the first occurrence of the catastrophe. In our work, we apply, to a fractional queueing model with catastrophes, the approach proposed by \cite{fila}. This allows us to obtain the state probabilities of the fractional queue with catastrophes starting from any number of customers; extending some of the results from \cite{catastrofef} where it is assumed that the process starts from the empty system. Besides, we obtained the mean and the variance from the respective probability generating function. Using the multinomial distribution, we proposed a moment estimator and a confidence interval for the parameters of the model. Finally, we performed computational simulations to validate the estimation of parameters, by showing that the estimators and the confidence intervals worked well for the chosen tests when we increase the sample size.  

\section*{Acknowledgements}
This study was financed in part by the Coordena\c{c}\~ao de Aperfei\c{c}oamento de Pessoal de N\'ivel Superior - Brasil (CAPES) - Finance Code 001. The authors thanks Katiane Silva Concei\c{c}\~ao and Juliana Cobre for fruitful discussions at the early stages of this work. Special thanks are also due to the two anonymous reviewers for their helpful comments and suggestions.


\end{document}